\documentclass[11pt,a4paper]{amsart}

\usepackage{latexsym}

\usepackage{amssymb}
\usepackage{amsthm}
\usepackage{amsmath}
\usepackage{amstext}
\usepackage{amscd}
\usepackage{color}

\numberwithin{equation}{section}

\theoremstyle{plain}
\newtheorem{thm}{Theorem}[section]

\newtheorem{cor}[thm]{Corollary}

\newtheorem{theorem*}{Theorem}[]

\theoremstyle{definition}

\theoremstyle{remark}
\newtheorem{rem}[thm]{Remark}

\newcommand{\R}{\mathbb{R}}

\newcommand{\Z}{\mathbb{Z}}

\DeclareMathOperator{\Ker}{Ker}
\DeclareMathOperator{\Coker}{Coker}
\DeclareMathOperator{\Image}{Im}

\begin{document}

\date{May 16, 2019}

\title[A local duality obstruction]{Real Intersection Homology II:\\ A local duality obstruction}

\author[C. McCrory and A. Parusi\'nski]{{Clint McCrory} \and {Adam Parusi\'nski}}
\thanks{Adam Parusi\'nski partially supported by ANR project LISA (ANR-17-CE40-0023-03)}

\address {Mathematics Department, University of Georgia, Athens GA
30602, USA}

\email{clint@math.uga.edu}

\address {Universit\'e Nice Sophia Antipolis, CNRS,  LJAD, UMR 7351, 06108 Nice, France}

\email{adam.parusinski@unice.fr}

\maketitle


We prove that the intersection pairing on the real intersection homology  of a real algebraic variety is not a dual pairing in general. A classical argument of Thom for manifolds, adapted to real intersection homology, shows that if this intersection pairing is nonsingular and $X$ is the link of a point in a real algebraic variety, with the dimension of $X$ even, then the intersection homology euler characteristic $I\chi(X)$ is even. Using an existence theorem of Akbulut and King \cite{AK} we show there is a singular algebraic surface $X$ that is the link of a point in a 3-dimensional real algebraic variety, such that $I\chi(X)$ is odd. Thus our definition of real intersection homology \cite{rih} does not have the key self-duality property suggested by Goresky and MacPherson (\cite{GM}, p.227), though it does enjoy their small resolution property.


\section {Real intersection homology}

We summarize the results of our preceding paper \cite{rih}. Let $X$ be a compact real algebraic variety of dimension $n$. The real intersection homology groups $IH_i(X)$, $i = 0,\dots, n$, are $\Z/2$ vector spaces invariant under arc-symmetric semialgebraic homeomorphisms. There are natural homomorphisms
$$
H^{n-i}(X;\Z/2)\to IH_i(X)\to H_i(X;\Z/2),\ i=0,\dots,n,
$$
and the composition is the classical Poincar\'e duality homomorphism. If $X$ is nonsingular and connected then these two maps are isomorphisms. If $\pi:\widetilde X\to X$ is a small resolution, then $\pi$ induces isomorphisms $H_i(\widetilde X)\cong IH_i(X)$ for all $i$. 

If $X$ is not compact, there are two versions of the real intersection homology of $X$, $IH^c_i(X)$ with compact supports and $IH^{cl}_i(X)$ with closed supports. The preceding results hold for compact or closed supports.

For all $i$ and $j$ with $i+j=n$ there is a bilinear intersection pairing
$$
\psi: IH^c_i(X)\times IH^{cl}_j(X) \to \Z/2,
$$
defined using the general position theorem of \cite{MPP}. If $X$ has isolated singularities then this pairing is nonsingular (\emph{i.e.} it is a dual pairing). In other words, the homomorphism $\Psi: IH^c_i(X) \to (IH^{cl}_j(X))^*$, $\Psi(\alpha)(\beta) = \psi(\alpha,\beta)$, is an isomorphism.


\section{The obstruction}

Real intersection homology is defined for a semialgebraic open subset of a real algebraic variety, such as the star of a point in a real algebraic variety. We prove that if the intersection pairing is nonsingular for all real algebraic varieties, then the intersection pairing is also nonsingular on all such stars.
And if the intersection pairing on the star of $v_0\in V$ is nonsingular, then $I_\chi(X)$ is even, where $X$ is the link of $v_0\in V$. 

First we present Thom's proof of the following classical result (\cite{T}, Th\' eor\` eme V.9, p.175), and then we discuss how to adapt his proof to our setting.

\begin{thm}\label{thom} Let $M$ be a compact even dimensional manifold without boundary. If $M$ is the boundary of a compact manifold $W$, then the euler characteristic $\chi(M)$ is even.\end{thm}

\begin{proof}
Let $H_*(M)$ be the singular homology of $M$ with $\Z/2$ coefficients. Let $n = 2k = \dim M$, and let $b_i = \dim H_i(M)$, the $i$th betti number of $M$, $i = 0,\dots,n$. By Poincar\' e duality, $b_i = b_{n-i}$ for all $i$. Thus
\begin{equation}
\chi(M) = \sum_i (-1)^i b_i = 2\sum _{i<k} (-1)^ib_i + (-1)^k b_k,
\end{equation}
so $\chi(M)$ is even if and only if $b_k$ is even.

To see that $b_k$ is even, consider the long exact homology sequence of the pair $(W,M)$,
\begin{equation}
\cdots \to H_{i+1}(W,M) \overset{\Delta}{\to} H_i(M) \overset{\alpha}{\to} H_i(W) \overset{\beta}{\to} H_i(W,M) \to \cdots ,
\end{equation}
and the corresponding exact sequence of dual vector spaces,
\begin{equation}
 \cdots \to H_{j}(W,M)^* \overset{\beta^*}{\to} H_{j}(W)^* \overset{\alpha^*}{\to} H_{j}(M)^* \overset{\Delta^*}{\to} H_{j+1}(W,M)^* \to \cdots .
\end{equation}
By Poincar\' e duality the intersection pairings on $M$ and $W$ induce an isomorphism of these two exact sequences. In particular we have the following commutative diagram, where the vertical isomorphisms are given by the intersection pairings:

\begin{equation}\label{ker-im}
\begin{CD}
H_{k+1}(W,M) @>\Delta>> H_k(M) @>\alpha>>  H_k(W) \\
 @VV\cong V @VV\cong V  @VV\cong V \\
H_{k}(W)^* @>\alpha^*>> H_k(M)^* @>\Delta^*>> H_{k+1}(W,M)^* 
\end{CD}
\end{equation}
\vskip.2in

\noindent By exactness of the rows we have $\Ker \alpha \cong \Image \alpha^*$.

Now since $\alpha$ and $\alpha^*$ are adjoint, the kernel of $\alpha$ is the annihilator of the image of $\alpha^*$, and $\Ker \alpha \cong \Coker \alpha^*$. Thus $\dim(\Ker\alpha) + \dim(\Image\alpha^*) = \dim(H_k(M))$, so $b_k = 2\dim (\Ker\alpha)$.
\end{proof}

\begin{rem} It's easy to see that if $a, b\in \Ker\alpha$ then the intersection $a\cdot b = 0$. (Let $a=[A]$, with $A=\partial A'$ and $b=[B] $ with $B=\partial B'$. If $A'$ and $B'$ are transverse, then $A \cdot B = \partial (A'\cdot B')$.) In other words, $\Ker \alpha\subset P$, where $P$ is the annihilator of $\Ker \alpha$ under the intersection product. Now $\dim (\Ker \alpha) + \dim P = b_k$, so $\dim P = \dim( \Ker \alpha)$, which implies that $\Ker \alpha = P$. Thus the inner product space $H_k(M)$ is \emph{split} (\cite{MH}, Chapter I, \textsection 6, p.12). It follows that the preceding argument applied to the rational homology of a compact oriented manifold $M$ of dimension $n = 4l$ gives Thom's theorem that the signature (index) is a cobordism invariant (\cite{T}, Corollaire V.11, p.176). (See also \cite{H}, Theorem 8.2.1 III, p.85.)
\end{rem}


Now we adapt this proof to real intersection homology. Consider the variety $X= L_\epsilon(V, v_0)$, the $\epsilon$-link of a point $v_0$ in a real algebraic subvariety $V$ of euclidean space: $L_\epsilon(V, v_0) = V\cap S(v_0, \epsilon)$, with $S(v_0, \epsilon)$ the sphere of radius $\epsilon$ about $v_0$, for $\epsilon$ sufficiently small.  We assume that $V$ has pure dimension $n+1$, so that $X$ has pure dimension $n$.

Let $\mathcal S$ be an algebraic stratification of $V$, and let $\rho:V\to \R$, $\rho(v) = |v-v_0|^2$. By \cite{PP}, Theorem 9.3, there is a stratification $\mathcal T$ of $\R$ such that the restriction of $\rho$ to the preimage of each stratum $T\in\mathcal T$ is a locally trivial fibration. 
Thus there exists $\eta>0$ such that for every $\epsilon$ with $0<\epsilon<\eta$, there exists $\delta>0$ and a stratum preserving arc-analytic semialgebraic homeomorphism
$$
\varphi: (\epsilon - \delta, \epsilon +\delta)\times L_\epsilon(V,v_0) \to \rho^{-1}((\epsilon -\delta)^2, (\epsilon +\delta)^2),
$$
with $\rho\varphi(t,x) = t^2$ and $\varphi(\epsilon, x) = x$. It follows that $L_\epsilon(V, v_0)$ is independent of $\epsilon$ up to stratified arc-analytic semialgebraic homeomorphism (\emph{cf.} \cite{PP}, Cor. 9.6). 

\begin{rem}
Since $\varphi$ is stratum preserving, the stratification induced by $\mathcal S$ on the set $\rho^{-1}((\epsilon -\delta)^2, (\epsilon +\delta)^2)$  corresponds to the product stratification on $(\epsilon - \delta, \epsilon +\delta)\times L_\epsilon(V,v_0)$. In other words, for every stratum $S\in\mathcal S$,
$$
\varphi^{-1}(S\cap \rho^{-1}((\epsilon -\delta)^2, (\epsilon +\delta)^2))
= (\epsilon - \delta, \epsilon +\delta)\times(S\cap L_\epsilon(V,v_0)).
$$
\end{rem}

Given such a pair $(\epsilon,\delta)$, let $\epsilon'=\epsilon+\delta$. 
We replace $(W,M)$ in the proof of Theorem \ref{thom}
 with $(N,X)$, with $X = L_\epsilon(V, v_0)$ and  $N = V\cap B(v_0, \epsilon')$, the open $\epsilon'$-star of $v_0$ in $V$, where $B(v_0, \epsilon')$ is the open ball of radius $\epsilon'$ about $v_0$.
 
\begin{thm}\label{les} There is a long exact sequence
\begin{equation}\label{les-ih}
 \cdots \to IH_{i+1}^{cl}(N) \overset{\Delta}{\to} IH_i(X) \overset{\alpha}{\to} IH^{c}_i(N) \overset{\beta}{\to} IH_i^{cl}(N) \to \cdots ,
\end{equation}
and the intersection pairings on $X$ and $N$ give a homomorphism from this exact sequence to the dual exact sequence
\begin{equation}\label{les-ih*}
 \cdots \to IH_{j}^{cl}(N)^* \overset{\beta^*}{\to} IH_{j}^c(N)^* \overset{\alpha^*}{\to} IH_{j}(X)^* \overset{\Delta^*}{\to} IH_{j+1}^{cl}(N)^* \to \cdots .
\end{equation}
\end{thm}

\begin{cor}
Suppose that $V$ is purely odd dimensional. If
the intersection pairings on the star $N$ and the link $X$ of $v_0\in V$ are nonsingular, then the intersection homology Euler characteristic $I_\chi(X)$ is even.
\end{cor}
For the corollary, note that if $\dim V = 2k+1$
we have a commutative diagram analogous to (\ref{ker-im}),
\begin{equation}
\begin{CD}
IH_{k+1}^{cl}(N) @>\Delta>> IH_k(X) @>\alpha>>  IH_k^c(N) \\
 @VV\cong V @VV\cong V  @VV\cong V \\
IH_{k}^c(N)^* @>\alpha^*>> IH_k(X)^* @>\Delta^*>> IH_{k+1}^{cl}(N)^* 
\end{CD}
\end{equation}
and we repeat Thom's argument  to conclude that $I\chi(X)$ is even.

\vskip.2in
To prepare for the proof of the theorem, recall that a semialgebraic $i$-chain is an equivalence class of closed semialgebraic sets of dimension $\leq i$, with respect to the relations $A+B \sim \operatorname{cl}((A\setminus B)\cup(B\setminus A))$ and $A\sim 0$ if $\dim A< i$. To simplify notation, we will identify a semialgebraic set with the chain it represents. The intersection homology groups are defined using allowable chains, which are represented by semialgebraic sets satisfying certain perversity conditions with respect to a good stratification (see \cite{rih}).

\begin{proof} First we define the sequence \eqref{les-ih}. Let $\mathcal S$ be a good stratification of $V$. The restriction of $\mathcal S$ to $N$ (resp.\ $X$) is a good stratification of $N$ (resp.\ $X$). If $C$ is an allowable $i$-chain of $X$ then $C$ is an allowable $i$-chain of $N$, and the homomorphism $\alpha:IH_i(X) \to IH^{c}_i(N)$ is induced by inclusion of chain groups. The map $\beta:IH^{c}_i(N) \to H_i^{cl}(N)$ is also induced by inclusion.

The  homomorphism $\Delta: IH_{i+1}^{cl}(N) \to IH_i(X)$ is defined as follows. 
If $c\in IH_{i+1}^{cl}(N)$, then $c$ is represented by an $(i+1)$-cycle $C$ such that $C$ is in general position with $X$. More precisely, for every stratum $S$ of $S|N$,
$$
\dim(C\cap X\cap S) \leq \dim(C\cap S) + \dim(X\cap S) - \dim S,
$$
and, since $S(v_0, \epsilon)$ is transverse to $S$, we therefore have
$$
\dim(C\cap X\cap S) \leq \dim(C\cap S) - 1 .
$$
Let $C_\epsilon = C\cap \{v\in V\ |\ |v-v_0|\leq \epsilon \}$. Then $\partial C_\epsilon$ is an allowable $(i-1)$-cycle of $X$ that represents $\Delta(c)$.

Next we show that the sequence \eqref{les-ih} is exact. Let $\pi:(\epsilon - \delta, \epsilon +\delta)\times X \to X$ be the projection.

(i) $\Image\Delta = \Ker\alpha$. Let $c\in IH_{i+1}^{cl}(N)$ be represented by an allowable cycle $C$ as above. Then $\Delta(c)$ is represented by $\partial C_\epsilon$. Now $\alpha(\Delta(c))$ is also represented by $\partial C_\epsilon$, and $C_\epsilon$ is compact, so $\alpha(\Delta(c))=0$. Thus $\Image\Delta \subset \Ker\alpha$. 
Conversely, suppose that $b\in IH_i(X)$ is represented by an allowable cycle $B$. If $\alpha(b)=0$, then by general position there is an allowable compact chain $C$ of $N$ such that $C$ is in general position with $X$ and $\partial C = B$. [See \cite{rih}, proof of Thm. 2.7. Suppose that $C_0$ is an allowable compact chain with $\partial C_0=B$. There is an arc-analytic, stratum-preserving semialgebraic homeomorphisms $\Phi_{t_0}:X\to X$ such that $\Phi_{t_0}(C_0)$ is in general position with $X$, and a $t_0$-deformation homology $C_1$ of $B$ such that $\partial C_1 = B + \Phi_{t_0}(B)$. Let $C = C_0+C_1$.] Let $D = C + \varphi_*([\epsilon,\epsilon')\times B)$. Then $D$ is an allowable cycle of $N$, in general position with $X$, and so $\partial(D_\epsilon)$ is in the image of $\Delta$. Let $D= D_\epsilon + D'_\epsilon$, where $D'_\epsilon$ is contained in $\varphi([\epsilon,\epsilon')\times X)$. Then $\partial(D_\epsilon) = \partial (D'_\epsilon)$. Let $D'_\epsilon = E + \varphi_*([\epsilon,\epsilon')\times B)$, so $\partial D'_\epsilon = \partial E + B$. Now $\partial E = \partial\varphi_*(\pi_*(\varphi_*^{-1}(E)))$, and $\varphi_*(\pi_*(\varphi_*^{-1}(E)))$ is a chain of $X$, so $\partial D'_\epsilon$ is homologous to $B$ in $X$. Thus $\Ker\alpha\subset \Image\Delta $.

(ii) $\Image\alpha = \Ker \beta$. Let $b\in IH_i(X)$ be represented by an allowable cycle $B$ in $X$. Then $\alpha(b)$ is represented by the same cycle $B$ in $N$, so $B$ also represents $\beta(\alpha(c))$. But $B = \partial \varphi_*([\epsilon, \epsilon')\times B)$, so $\beta(\alpha(b))=0$. Thus $\Image\alpha \subset \Ker \beta$.
Conversely, suppose that $c\in IH^c_i(N)$ is represented by a compact allowable cycle $C$. Choose $\delta' >0$ so that $\delta'<\delta$ and $C$ is contained in  $\{v \in V\ |\ |v-v_0|< \epsilon+\delta'\}$. If $\beta(c)= 0$ there is an allowable chain $D$ of $N$ such that $\partial D = C$ and $D$ is in general position with $L_{\epsilon+\delta'}(V, v_0)$. Let $C'= \partial(D\cap \{v\in V\ |\ |v-v_0|\geq \epsilon + \delta'\})$. Then 
$\partial(D\cap \{v\in V\ |\ |v-v_0|\leq \epsilon + \delta'\})= C+C'$, so $C$ is homologous to $C'$. Let $E= \varphi_*([\epsilon, \epsilon+\delta')\times \varphi_*^{-1}C')$. Then $\partial E = C' + C''$, with $C''$ contained in $X$. So $C$ is homologous to a cycle in $X$, and  we have $\Ker\beta\subset \Image\alpha$.

(iii) $\Image\beta = \Ker\Delta$. Let $c\in IH_i^c(N)$ be represented by a compact allowable cycle $C$ of $N$ such that $C$ is in general position with $X$. Then $\beta(c)$ is also represented by $C$, and so $\Delta(\beta(c))$ is represented by $\partial C_\epsilon$ as above. Now $\partial C_\epsilon = \partial \operatorname{cl} (C\setminus C_\epsilon) = \partial\varphi_*(\pi_*\varphi_*^{-1}(\operatorname{cl} (C\setminus C_\epsilon)))$. Thus $\Delta(\beta(c))=0$, and $\Image\beta \subset \Ker\Delta$.
Conversely, suppose $c\in IH^{cl}_{i+1}(X)$ is represented by an allowable cycle $C$ in general position with $X$, so $\Delta(c)$ is represented by the cycle $\partial C_\epsilon$. If $\Delta(c) = 0$ there is an allowable chain $B$ in $X$ with $\partial B = \partial C_\epsilon$. Thus the compact chain $B+C_\epsilon$ is an allowable cycle. We claim that $B+C_\epsilon$ is homologous to $C$, so $c\in\Image\beta$, and we have $\Ker\Delta \subset \Image\beta$. 

To prove the claim, consider the cycle $D= C + (B+C_\epsilon) = B + \operatorname{cl}(C\setminus C_\epsilon)$, and the chain $E$ in $(\epsilon-\delta,\epsilon +\delta)\times(\epsilon-\delta,\epsilon +\delta)\times X$ given by 
$E = \{(s,t,x)\ |\ \varphi(t,x)\in D),\ s\geq t \} $.
If $p:(\epsilon-\delta,\epsilon +\delta)\times(\epsilon-\delta,\epsilon +\delta)\times X
\to N$, $p(s,t,x) = \varphi(s,x)$, then the restriction $p|\{(s,t,x)\ |\ s\geq t\}$ is proper, and $\partial p_*(E) = D$.

Finally we show that the maps from \eqref{les-ih} to \eqref{les-ih*}  given by the intersection pairings form a commutative ladder ($i+j=n$):
\begin{equation}\label{ladder}
\begin{CD}
IH_{i+1}^{cl}(N) @>\Delta>> IH_i(X) @>\alpha>>  IH_i^c(N)@>\beta>>  IH_i^{cl}(N) \\
 @VVV @VVV  @VVV@VVV\\
IH_{j}(N)^* @>\alpha^*>> IH_j(X)^* @>\Delta^*>> IH_{j+1}^{cl}(N)^* @>\beta^*>>  IH_{j+1}^c(N)^*
\end{CD}
\end{equation}

(i) The right-hand square of \eqref{ladder} commutes if for all $a\in IH^c_i(N)$ and $b\in IH^c_{j+1}(N)$, we have $a\cdot \beta(b) = \beta(a) \cdot b$. Suppose that $a$ is represented by a compact allowable cycle $A$, and $b$ is represented by a compact allowable cycle $B$, with $A$ (stratified) transverse to $B$. (See \cite{rih}. This means that $A$ has a semialgebraic stratification $\mathcal A$, and $B$ has a semialgebraic stratification $\mathcal B$, such that $A$ and $B$ are substratified objects of the good stratification $\mathcal S|N$, and for all strata $S\in\mathcal A$, $T\in\mathcal B$, with $S$ and $T$ contained in a stratum $U\in\mathcal S|N$, we have that $S$ and $T$ are transverse in $U$.) Then $A$ and $B$ intersect in a finite number $m$ of points in top strata of $\mathcal S|N$. Now $A$ also represents $\beta(a)$, and $B$ represents $\beta(b)$. Thus $a\cdot \beta(b) = m = \beta(a) \cdot b$.

(ii) The middle square of \eqref{ladder} commutes if for all $a\in IH_i(X)$ and $b\in IH_{j+1}^{cl}(N)$, we have $a\cdot \Delta(b) = \alpha(a) \cdot b$. Suppose that $a$ is represented by an allowable cycle $A$ of $X$, and $b$ is represented by an allowable cycle $B$ of $N$, with $A$ transverse to $B$ in $N$. We may assume also that $B$ is transverse to $X$ in $N$, which implies that $B\cap X$ represents $\Delta(b)$ and $B\cap X$ is transverse to $A$ in $X$. Let $m$ be the number of points of $A\cap B$. Then  $a\cdot \Delta(b) = m = \alpha(a) \cdot b$.

(iii) The left-hand square of \eqref{ladder} commutes if for all $a\in IH_{i+1}^{cl}(N)$ and $b\in IH_j(X)$, we have $a\cdot \alpha(b) = \Delta(a) \cdot b$. This is the same assertion as (ii), with $a$ and $b$ interchanged.
\end{proof}

\begin{thm} If the intersection pairing is nonsingular for the pure-dimensional algebraic variety $V$ and for the link of a point $v_0$ in $V$, then the intersection pairing on the open star of $v_0$ in $V$ is nonsingular. More precisely, if the intersection pairing is nonsingular for $V$ and for $L_\epsilon(V,v_0)$ for small $\epsilon>0$, then it is nonsingular for $N_{\epsilon+\delta}=V\cap B(v_0,\epsilon+\delta)$ for $\delta>0$ sufficiently small.
\end{thm}

\begin{rem}
Presumably the $\eta$-star $N_\eta = V\cap B(v_0,\eta)$ is independent of $\eta$ for small $\eta>0$, up to stratified arc-analytic semialgebraic homeomorphism. We do not need this fact.
\end{rem}

\begin{proof} We sketch the proof; the details are similar to the proof of Theorem \ref{les}.
Choose $\epsilon$ and $\delta$ as in the proof of Theorem \ref{les}. As before, let $N = V\cap B(v_0,\epsilon+\delta)$, and $X= L_\epsilon(V,v_0)$. Now let $N'= V\setminus\operatorname{cl}B(v_0, \epsilon -\delta)$. There is a Mayer-Vietoris sequence
\begin{equation}\label{mv-ihc}
 \cdots \to IH_{i+1}^{c}(V) \overset{\Delta}{\to} IH_i^c(N\cap N') \overset{\alpha}{\to} IH^{c}_i(N) \oplus IH^{c}_i(N') \overset{\beta}{\to} IH_i^{c}(V) \to \cdots ,
\end{equation}
such that the intersection pairings map \eqref{mv-ihc} to the dual of the Mayer-Vietoris sequence
\begin{equation}\label{mv-ihcl}
 \cdots \to IH_{j}^{cl}(V) \overset{\beta'}{\to} IH_{j}^{cl}(N)\oplus IH_{j}^{cl}(N') \overset{\alpha'}{\to} IH_{j}^{cl}(N\cap N') \overset{\Delta'}{\to} IH_{j-1}^{cl}(V) \to \cdots .
\end{equation}

Furthermore, for all $i, j$ with $i+j=\dim V-1$ there is a commutative diagram
\begin{equation}
\begin{CD}
IH_{i}^{c}(N\cap N') @>\cong>> IH_i(X)\\
 @VVV @VVV   \\
IH_{j+1}^{cl}(N\cap N')^* @>\cong>> IH_j(X)^* 
\end{CD}
\end{equation}
where the vertical arrows are given by the intersection pairings. So if the intersection pairing on the variety $X$ is nonsingular, so is the intersection pairing on the semialgebraic set $N\cap N'$.

Thus, if the intersection pairing is nonsingular on $V$ and $X$, then by the Five Lemma, the intersection pairings on $N$ and $N'$ are nonsingular.
\end{proof}


\section{A counterexample}

We show there is a compact real algebraic variety $X$ of pure dimension 2 such that $X$ is the link of a point in a real algebraic variety of dimension 3, and the real intersection homology euler characteristic $I\chi(X)$ is odd. 
The variety $X$ we seek will be homeomorphic to the topological space obtained from the real projective plane by identifying two points. 

Let $Z$ be a genus one complex projective curve with one node, considered as a real algebraic surface. Let $Y$ be obtained as the (real) blow-up of a nonsingular point of $Z$. 
The variety $Y$ has euler characteristic 0, and all links in $Y$ have euler characteristic 0, so all five Akbulut-King numbers of $Y$ are zero (\emph{cf.}\ \cite{MP}, p.80). Therefore there exists a variety $X$ homeomorphic to $Y$ such that $X$ is the link of a point in a 3-dimensional real algebraic variety. More precisely, by Theorem 7.1.2 of \cite{AK} applied to the suspension of $Y$, there exists such an algebraic variety $X$ whose singular stratification (\cite{AK}, p.31) has just three strata: $\{x_0\}$, $\{x_1\}$, and $X\setminus \{x_0,x_1\}$, where $x_0$ is the topological singular point. (On page 197 of \cite{AK} this argument is given for $Y$ the disjoint union of the projective plane and a point.)

We claim there exists a semialgebraic homeomorphism $h:X\to Y$. This follows from the Hauptvermutung for 2-dimensional simplicial complexes (\cite{B}, Thm.\ 4.6, p.190). [A semialgebraic triangulation of $X$ (resp.\ $Y$) is a semialgebraic homeomorphism from $X$ (resp.\ $Y$) to a finite euclidean simplicial complex $K$ (resp.\ $L$). If $X$ and $Y$ are homeomorphic then $K$ and $L$ are homeomorphic, so by the Hauptvermutung $K$ and $L$ have isomorphic simplicial subdivisions. So $K$ and $L$ are piecewise linearly homeomorphic, hence semialgebraically homeomorphic, so $X$ and $Y$ are semialgebraically homeomorphic.]

We will compute the real intersection homology groups of $X$ using the semialgebraic homeomorphism $h$. If $y_0\in Y$ is the topological singular point then $h(x_0) = y_0$. Let $y_1 = h(x_1)$. Consider the homomorphism $h_k:IH_k(X)\to H_k(Y)$, $k=0,1,2$, induced by the semialgebraic homeomorphism $h:X\to Y$. In other words $h_k$ is induced by the chain map that takes the allowable $k$-chain $C$ of $X$ to the $k$-chain $h(C)$ of $Y$. 

Let $\mathcal S$ be a good algebraic stratification (\cite{rih} \S 1) of the 2-dimensional variety $X$ such that $\mathcal S$ is a refinement of the singular stratification $\{ \{x_0\}, \{x_1\}, X\setminus \{x_0,x_1\}\}$ of $X$. An $\mathcal S$-allowable 0-chain of $X$ is a finite set of points contained in codimension 0 strata of $\mathcal S$. An $\mathcal S$-allowable 1-chain is a semialgebraic 1-chain $C$ such that $\dim(C\cap S)\leq 0$ and $\Sigma C\cap S= \emptyset$ for every codimension 1 stratum $S\in \mathcal S$, and $C\cap T = \emptyset$ for every codimension 2 stratum $T\in \mathcal S$. (Since $\partial C\subset \Sigma C$, it follows that $\partial C$ is allowable.) An $\mathcal S$-allowable 2-chain is a semialgebraic 2-chain $C$ such that $\dim(\Sigma C\cap S)\leq 0$ for every codimension 1 stratum $S\in \mathcal S$, and $\partial C$ is $\mathcal S$-allowable.

An allowable $0$-chain $\{p,q\}$ of $X$ bounds a semialgebraic 1-chain $C$ in $X\setminus \{x_0,x_1\}$, so by general position for nonsingular varieties (\cite{rih}, Prop.\ 3.2),  $\{p,q\}$ bounds an allowable 1-chain. Thus $h_0:IH_0(X)\to H_0(Y) = \Z/2$ is an isomorphism. The only semialgebraic 2-cycle of $X$ is the fundamental class of $X$, which is allowable. Thus  $h_2:IH_2(X)\to H_2(Y) = \Z/2$ is an isomorphism.

Now $H_1(Y)$ has dimension 2 with basis $\{a, b \}$, where $a$ is represented by the exceptional divisor of the blowup $\pi:Y\to Z$, and $\pi_*(b)\neq 0$. We claim that $h_1:IH_1(X)\to H_1(Y)$ is an isomorphism onto the span of $a$, so $IH_1(X) = \Z/2$.

First note that $h_1$ is injective. If $C$ an allowable 1-cycle in $X$, and $h(C)$ bounds a 2-chain $D$ of $Y$, then $C = \partial h^{-1}(D)$. But $h^{-1}(D)$ is an allowable 2-chain of $X$. 

Next we show that the image of $h_1$ is spanned by $a$. If $C$ is an allowable 1-cycle of $X$, then $C\cap \{x_0\} = \emptyset$. But for every 1-cycle $B$ of $Y$ representing $b$, we have $B\supset \{y_0\}$. Since $h^{-1}\{y_0\} = \{x_0\}$, we conclude that $b$ is not in the image of $h_1$. Let $A$ be a 1-cycle of $Y$ such that $A$ represents $a$ and $A\cap\{y_0,y_1\} = \emptyset$. Again by general position for the nonsingular variety $X\setminus \{x_0,x_1\}$, the cycle $C=h^{-1}(A)$ is homologous to an allowable 1-cycle $C'$, and so $h(C')$ is homologous to $a$. Thus $a$ is in the image of $h_1$.  

We have shown that $IH_0(X) = \Z/2$, $IH_1(X) = \Z/2$, and $IH_2(X) = \Z/2$, so the intersection homology euler characteristic $I\chi(X) = 1$.


\end{document}